\documentclass[11pt, reqno]{amsart}
\usepackage[dvipsnames]{xcolor}
\usepackage{mathtools}
\mathtoolsset{showonlyrefs}
\usepackage{etoolbox}

\usepackage{blindtext}
\usepackage{latexsym}
\usepackage[pagewise]{lineno}
\usepackage[utf8]{inputenc}
\usepackage{exscale}
\usepackage{amsmath}
\usepackage{amssymb}
\usepackage{amsfonts}
\usepackage{mathrsfs}
\usepackage{amsbsy}
\usepackage{relsize}

\definecolor{mypink1}{rgb}{0.858, 0.188, 0.478}
\definecolor{mypink2}{RGB}{219, 48, 122}
\definecolor{mypink3}{cmyk}{0, 0.7808, 0.4429, 0.1412}
\definecolor{mygray}{gray}{0.6}
\definecolor{venetianred}{rgb}{0.78, 0.03, 0.08}
\definecolor{sapphire}{rgb}{0.03, 0.15, 0.4}
\definecolor{utahcrimson}{rgb}{0.83, 0.0, 0.25}
\definecolor{trueblue}{rgb}{0.0, 0.45, 0.81}
\definecolor{carminered}{rgb}{1.0, 0.0, 0.22}
\definecolor{cobalt}{rgb}{0.0, 0.28, 0.67}
\definecolor{cornflowerblue}{rgb}{0.39, 0.58, 0.93}

\usepackage{hyperref}
\hypersetup{
setpagesize=false,
 bookmarksnumbered=true,%
 bookmarksopen=true,%
 colorlinks=true,%
 linkcolor=utahcrimson,
 citecolor=blue
}
\usepackage{comment}
\PassOptionsToPackage{reqno}{amsmath}
\usepackage{esint} 
\usepackage{amssymb} 
\usepackage{stmaryrd}
\usepackage{cite} 

\usepackage{tikz}

\usepackage{mathtools}
\usepackage{soul}
\usepackage{color}

\usepackage{slashed}

\newcommand{\R}{\ensuremath{\mathbb{R}}}

\newcommand{\N}{\ensuremath{\mathbb{N}}}
\newcommand{\bm}{\mathbf{m}}
\newcommand{\bX}{\mathbf{X}}

\newcommand{\mC}{\mathcal{C}}
\newtheorem{thm}{Theorem}[section]

\newtheorem{Lem}[thm]{Lemma}
\newtheorem{Prop}[thm]{Proposition}
\newtheorem{Def}{Definition}[section]
\newtheorem{Rem}{Remark}[section]
\newtheorem{rems}{Remarks}[section]

\title[Inhomogeneous semilinear wave equation]{Well-posedness and linearization for a semilinear wave equation  with spatially growing nonlinearity}
\author[D. Draouil, M. Majdoub]{Dhouha Draouil and  Mohamed Majdoub}
\address[D. Draouil]{Universit\'e de Tunis El Manar, Facult\'e des Sciences de Tunis, D\'epartement de Math\'ematiques, Laboratoire \'equations aux d\'eriv\'ees partielles (LR03ES04), 2092 Tunis, Tunisie.}
\address[M. Majdoub]{Department of Mathematics, College of Science, Imam Abdulrahman Bin Faisal University, P. O. Box 1982, Dammam, Saudi Arabia.\newline Basic and Applied Scientific Research Center, Imam Abdulrahman Bin Faisal University, P.O. Box 1982, 31441, Dammam, Saudi Arabia.}
\email{\sl \textcolor{blue}{douhadraouil@yahoo.fr}}
\email{\sl \textcolor{blue}{mmajdoub@iau.edu.sa}}
\email{\sl \textcolor{blue}{med.majdoub@gmail.com}}

\begin{document}

\medskip
\begin{abstract}
{We study the initial value problem for a defocusing semi-linear wave equation with spatially growing nonlinearity. By employing Moser-Trudinger-type inequalities and Strichartz estimates, we establish global well-posedness in the energy space for radially symmetric initial data. Furthermore, we derive the linearization of energy-bounded solutions using the methodology introduced in \cite{P.GR}. The main challenge in our analysis arises from the spatial growth of the nonlinearity at infinity, which prevents the direct application of Sobolev embeddings or Hardy inequalities to control the potential energy. The main novelty of this work lies in overcoming this challenge within the radial framework through the combined application of the Strauss inequality and Strichartz estimates.}
\end{abstract}

\maketitle

\section{Introduction and main results}
\label{S1}

This paper is concerned with the Cauchy problem associated to the following semi-linear
 wave equation
\begin{equation}
\label{eq a1}\partial_{t}^{2}u-\Delta_{x}u+\mathbf{m}\, u =|x|^bf(u)
,\end{equation}
where
$u:= u(t,x)$ is a real-valued function of
$(t,x)\in \mathbb{R}\times {\mathbb{R}}^{N}$, $\mathbf{m}\geq 0$, $b\in\R$, and {$f: \R\to\R$ is a continuous function.}
We assume that the initial data 
\begin{equation}
\label{eq aC.D}
(u_0,u_1):=(u(0),\partial_{t}u(0))
\end{equation}
 are radially symmetric and belong to the energy space $H^1(\R^N)
\times L^{2}(\mathbb{R}^N)$.

Equation  \eqref{eq a1} has the conserved energy
\begin{equation}
\label{energy}
\begin{split}
E(u,t)&=\frac{1}{2}\left(\|\partial_{t}u(t)\|_{L^2}^2+\|\nabla\,u(t)\|_{L^2}^2+\bm \|u(t)\|_{L^2}^2\right)\\&-\int_{\R^N}\,|x|^b\,F(u(t,x))\,dx,
\end{split}
\end{equation}
where $F(u)=\int_0^u\,f(v)\,dv.$ 

The nonlinear wave equation is a fundamental equation in mathematical physics that describes the behavior of waves in certain physical systems. It has been the subject of extensive study and research in the past several decades.

{For} this equation, the term {\it defocusing} refers to the behavior of the waves, indicating that they tend to spread out and disperse rather than concentrating or focusing at specific points. This is in contrast to the {\it focusing} case, where waves tend to concentrate and exhibit a more localized behavior.

The equation is classified as a semi-linear wave equation when it involves a linear wave operator combined with a non-linear term. The linear wave operator captures the wave propagation behavior, while the non-linear term introduces non-linear interactions that can arise in various physical phenomena.

Before proceeding further, let us recall some historical context regarding this problem. We start by considering the defocusing semi-linear wave equation in $N\geq 3$ spatial dimensions, expressed as
\begin{equation}
\label{eq a2}\partial_{t}^{2}u-\Delta_{x}u+ |u|^{p-1}u=0,\quad p>1.\end{equation}

The well-posedness of \eqref{eq a2} within the Sobolev spaces $H^s$ has been extensively explored, as documented in various works such as \cite{GI.SO.VE, GI.VE 1, M.GR 2, L.V.K 2, SH.ST 1, SH.ST 2}. It is established that the Cauchy problem associated with \eqref{eq a2} is locally well-posed in the standard Sobolev space $H^s({\mathbb{R}^N})$ under the condition $s > \frac{N}{2}$, or when $\frac{1}{2} \leq s < \frac{N}{2}$ and $p \leq 1 + \frac{4}{N-2s}$. Furthermore, if $p = 1 + \frac{4}{N-2s}$ and $\frac{1}{2} \leq s < \frac{N}{2}$, there are global $H^s$ solutions for small Cauchy data \cite{NA.OZ}.

The global solvability within the energy space has a rich historical background. A pivotal value of the {power} $p$ emerges, denoted as $2_N^* := \frac{N+2}{N-2}$, leading to three main scenarios. In the subcritical case $(p < 2_N^*)$, Ginibre and Velo established in \cite{GI.VE 1} the global existence and uniqueness in the energy space.

For the critical case $(p = 2_N^*)$, global existence was initially obtained by Struwe in the radially symmetric setting \cite{M.STR 1}, followed by Grillakis \cite{M.GR 1} in the general case, and subsequently by Shatah-Struwe \cite{SH.ST 2} in various dimensions.

Regarding the supercritical case $(p > 2_N^*)$, the question remains unresolved except for some partial results (refer to \cite{G.LE, G.LEB}). {Recently, it was shown in \cite{Colombo} that \eqref{eq a2} is globally well-posed for large data in three-dimensional space within a slightly supercritical regime, specifically for $ 5 \leq p < 5 + \delta_0 $, where $ \delta_0 > 0 $.}

In two spatial dimensions, any polynomial nonlinearity is regarded as subcritical with respect to the $H^1$ norm. Therefore, it is reasonable to explore an exponential nonlinearity. Notably, Nakamura and Ozawa \cite{NA.OZ} established global well-posedness and scattering for small Cauchy data in any spatial dimension $N \geq 2$. Subsequently, {Atallah Baraket} \cite{A} demonstrated a local existence result for the equation
\begin{equation}
\label{eq a3}
  \partial_{t}^{2}u - \Delta_{x}u + u{\rm e}^{\alpha u^{2}} = 0,\tag{$E_{\alpha}$}
\end{equation}
for $0 < \alpha < 4\pi$ and radially symmetric initial data $(0, u_1)$ with compact support.
{In the subsequent works \cite{IB. MAJ. MAS 2, Duke}, global well-posedness and the asymptotic behavior in the energy space were successfully achieved for both subcritical and critical regimes.} Recall that the Cauchy problem $(E_{4\pi})$-\eqref{eq aC.D} is said to be subcritical if
$$E^0:=\|u_1\|_{L^2(\R^2)}^2+\|\nabla u_0\|_{L^2(\R^2)}^2+\frac{1}{4\pi}\int_{\R^2}\Big({\rm e}^{4\pi u_0^2}-1\Big) dx <1.$$
It is critical if $E^0=1$ and supercritical if $E^0>1$.
{It is noteworthy to emphasize that the supercritical case was thoroughly investigated in \cite{IB.MAJ.MAS 3, APDE-2011}. 
For nonlinearities with weaker growth than ${\rm e}^{u^2}$, well-posedness, stability, and blow-up investigations were carried out in \cite{JPDE, Tar1} and \cite{Tar2, Tar3}.}\\

{We now return to the equation \eqref{eq a1}. For the case where $ b \in (-1, 0) $, $ N = 2 $, and $ f(u) = \pm|u|^{p-1}u $ or $ f(u) = \pm u\left({\rm e}^{\alpha |u|^2} - 1\right) $ with $ p > 1 $ and $ 0 < \alpha < 4\pi(1 + b) $, this problem was previously investigated in \cite{INW2015}. Furthermore, in \cite{Georgiev1}, the global existence and blow-up behavior of solutions to \eqref{eq a1} with focusing nonlinearity $ f(u) = |u|^{p-1}u $ and $ -2 < b < 0 $ were examined in both subcritical and critical regimes. For further investigations in the focusing regime, particularly those involving a general singular weight $ V(x) $, we refer the reader to \cite{Georgiev2}. }

{We note that equation \eqref{eq a1} can be interpreted as the relativistic counterpart of the inhomogeneous Schr\"odinger equation:
\begin{equation}
    \label{INLS}
    {\rm i}\partial_{t} u+\Delta_{x}u=|x|^bf(u).
\end{equation}
In the framework of quantum mechanics, the Klein-Gordon equation is regarded as the relativistic version of the Schr\"odinger equation, as it extends the principles of quantum mechanics to incorporate relativistic effects. While the Schr\"odinger equation is confined to non-relativistic systems, the Klein-Gordon equation offers a consistent description of relativistic quantum systems, particularly for spin$-0$ particles. For a more comprehensive insights of this topic, see \cite{Quantum, Quantum1}.\\
Equation \eqref{INLS} models the propagation of laser beams in specific plasma media and nonlinear optical systems. For further details, we refer the reader to \cite{LT, Gi} and the references therein.}

{The primary challenge in analyzing \eqref{eq a1} arises from the spatial growth of $|x|^b$ at infinity for $b>0$, which hinders the straightforward application of Sobolev embedding or Hardy inequality to control the potential energy for energy-solutions. The key innovation of this work lies in addressing this issue within the radial framework by making use of the Strauss inequality in combination with Strichartz estimates.}

 We are interested in strong solutions of \eqref{eq a1} for $b>0$ in dimensions 
 $N=2$ and $N=3$.  Specifically, for the three-dimensional we assume $\bm=0$, with the nonlinearity $f$ given by
 \begin{equation}
 \label{3D}
 f(u)=-|u|^{p-1}u\;\;(p>1).
 \end{equation}
In the two-dimensional case, we assume $\bm=1$, and  
\begin{equation}
\label{2D}
f(u)=-\left({\rm e}^u-1-u\right).
 \end{equation}
{Note that the energy defined in \eqref{energy} takes the form  
\begin{equation*}
E(u,t) = \frac{1}{2} \left( \|\partial_{t} u(t)\|_{L^2}^2 + \|\nabla u(t)\|_{L^2}^2 \right) + \frac{1}{p+1} \int |x|^b |u(t,x)|^{p+1} \, dx,
\end{equation*}
when the nonlinearity is given as in \eqref{3D}. In contrast, for the nonlinearity defined in \eqref{2D}, the energy becomes  
\begin{equation*}
\begin{split}
E(u,t) &= \frac{1}{2} \left( \|\partial_{t} u(t)\|_{L^2}^2 + \|\nabla u(t)\|_{L^2}^2 + \|u(t)\|_{L^2}^2 \right)\\& + \int |x|^b \left( e^{u(t,x)} - 1 - u(t,x) - \frac{u^2(t,x)}{2} \right) dx.
\end{split}
\end{equation*}}
 
Equation \eqref{eq a1} with $f$ as specified in \eqref{3D} exhibits a scaling invariance. Specifically, if $u$ solves \eqref{eq a1}, then $u_\lambda$ also does, where
	$$
	u_\lambda(t,x):= \lambda^{\frac{2+b}{p-1}} u(\lambda t, \lambda x), \quad \lambda>0.
	$$
	A straightforward calculation yields
	$$
	\|u_\lambda(0)\|_{\dot{H}^s} = \lambda^{s+\frac{2+b}{p-1}-\frac{N}{2}} \|u_0\|_{\dot{H}^s}
	$$
	indicating that the critical regularity space is $\dot{H}^{s_c}$, where $s_c:= \frac{N}{2}-\frac{2+b}{p-1}.$
	{The case $s_c=1$ corresponds to the energy-critical nonlinearity $p=\frac{N+2+2b}{N-2}$.}

Our primary objective is to establish global well-posedness in the energy space, without imposing any restrictions on the size of the Cauchy data, in the defocusing case. In two dimensions, a key advantage arises from the fact that the nonlinearity is comparatively weaker than ${\rm e}^{\alpha |u|^2}$ for any $\alpha>0$. As a result, the energy estimate and a set of technical tools are adequate to complete the proof.

However, in dimensions three (or higher), we will need to assume that the nonlinearity is sub-critical in a specific manner, which will be precisely defined later. This assumption is {somehow technical} to address the heightened complexity arising from the nonlinearity.\\

{In the second part of this work, our aim is to explore the concept of {\it linearizability} for \eqref{eq a1} by showing, under appropriate assumptions, that the nonlinear term $f(u)$ does not induce extra oscillations or energy concentrations. For a rigorous formulation of this concept, refer to Definition \ref{add} below.}

The study of linearization primarily originated from the investigation of equations featuring a polynomial-type nonlinearity in three spatial dimensions, as exemplified by equation \eqref{eq a2}. Key references in this field include \cite{BH.GR,BMM, P.GR, OMTS, M.MAJ}. These studies have made substantial contributions to our comprehension of linearization techniques and have set the groundwork for subsequent advancements in the field.

{To provide a concrete example of the concept of {\it linearizability}, we examine the linear wave equation associated with \eqref{eq a2}:}
\begin{equation}\partial_{t}^{2}u-\Delta_{x}u=0.\label{eq a14}\end{equation}
Let $(\varphi_n,\psi_n)$ be a bounded sequence in $H^1({\mathbb{R}}^3)\times L^2({\mathbb{R}}^3)$ such that
\begin{equation}
(\varphi_n,\psi_n)\rightharpoonup (0,0)\quad\hbox{in}\quad  H^1({\mathbb{R}}^3)\times L^2({\mathbb{R}}^3),\label{eq a15}
\end{equation}
\begin{equation}
\exists\; {R_0}>0,\quad\forall\; n,\quad \hbox{supp}(\varphi_n)\cup\hbox{supp}(\psi_n)\subset \{x\in {\mathbb{R}}^{3},\quad |x|\leq{R_0}\}.\label{eq a16}
\end{equation}
We denote by $(u_n)$, (respectively $(v_n)$) the sequence of
finite energy
{solutions} to \eqref{eq a2}), (respectively \eqref{eq a14}) satisfying
$$(u_n(0),\partial_t{u_n}(0))=(v_n(0), \partial_t{v_n}(0))=(\varphi_n,\psi_n).$$
\begin{Def}\label{add}
Let $T$ a positive real number. We say that the sequence $(u_n)$
is linearizable on $[0,T]$ if
\begin{equation}
\sup_{t\in[0,T]}\Big(\|\partial_{t}(u_n-v_n)(t)\|_{L^2}^2+\|\nabla(u_n-v_n)(t)\|_{L^2}^2\Big)\longrightarrow 0\quad\mbox{as}\quad
n\rightarrow\infty.
\end{equation}
\end{Def}
In the work by G\'erard \cite{P.GR}, it was demonstrated that in the subcritical case ($p<5$), the sequence $(u_n)$ exhibits linearizability within every time bounded interval. However, in the critical case ($p=5$), the behavior of the solution can manifest nonlinearity, primarily attributable to the lack of compactness in the Sobolev embedding ${\dot{H}}^1({\mathbb{R}}^3)\hookrightarrow L^{6}({\mathbb{R}}^3)$.
In the aforementioned case, G\'erard established the linearization under the following necessary and sufficient condition:
\begin{equation*}
\|v_n\|_{L^\infty([0,T],L^{6}({\mathbb{R}}^3))}\longrightarrow
0\quad\mbox{as}\quad n\rightarrow\infty.
\end{equation*}
A similar result was proved by {Ibrahim and Majdoub} in
\cite{IB.MA 2} for the equation
\begin{equation*}
\partial_{t}^{2}u-\mbox{div}\left(A(x)\nabla_{x}u\right) +|u|^{p-1}u=0\quad
\quad \mbox{in}\quad {\mathbb{R}}_{t}\times{\mathbb{R}}_{x}^{N},
\end{equation*}
where $N\geq 3,1<p\leq 2_N^*$ and {$A$ is a regular function with values in the space of positive definite matrices of size $N\times N$}, which is the identity outside a compact set of $\mathbb{R}^N$. We also refer to \cite{M.MAJ} for more general nonlinearities.\\

\subsection{Main results}
Our first results consist of the following well-posedness theorems, which are established through the application of a classical fixed-point argument.
\begin{thm}\label{main2D} Let $N=2$, $0\leq b\leq 1/2$, {$\bm=1$}, and suppose that the nonlinearity $f$ is given by \eqref{2D}. Then, for any radially symmetric data $(u_0,u_1)\in H^{1}\times L^{2}$, the
initial value problem \eqref{eq a1}-\eqref{eq aC.D} has a unique global solution $u$ in
the space
${\mathcal C}(\R,H^{1})\cap
{\mathcal C}^{1}(\R,L^{2}).$
Moreover, $u$ satisfies the energy identity
\begin{equation}\label{eq a11}
 E(u,t)=E(u,0), {\quad\mbox{\rm for all}\,\,t\in\R}.
\end{equation}
\end{thm}
\begin{rems}
~{\rm
 \begin{itemize}
 \item[(i)] {The choice of a spherically symmetric framework is dictated by the presence of the radial weight $|x|^b$. Removing this assumption remains a challenging and open question.} 
 \item[(ii)] {The assumption $0 \leq b \leq \frac{1}{2}$ is purely technical and stems from the limitations of our approach. We believe that this restriction can be relaxed.}
\item[(iii)] In \cite{Struwe}, the well-posedness for arbitrary smooth initial data was shown specifically for the case $b=0$
and for critical exponential nonlinearities. We conjecture that, for smooth radially symmetric initial data, global well-posedness holds provided that $0\leq b\leq 1/2$.
\item[(iv)] {In the two-dimensional case, we focus on the Klein-Gordon equation (corresponding to $\bm = 1$), to avoid the use of the homogeneous Sobolev space $\dot{H}^1(\R^2)=\dot{W}^{1,2}(\R^2)$, known for showing specific anomalous properties. Notably, $\dot{H}^1(\R^2)$ is not a subset of $\mathcal{D}'(\R^2)$, the space of distributions on $\R^2$. See, for example, \cite[Remarque 4.1, p. 319]{Lions} and \cite[Proposition 1.34, p. 26]{BCD}. For a comprehensive discussion on the characterization and properties of homogeneous Sobolev spaces, we refer to \cite{Brasco} and the references provided therein.}
\end{itemize}}
\end{rems}

\begin{thm}\label{main3D} Let $N= 3$, $b>0$, $1+b\leq p<4+b$, {$\bm=0$}, and suppose that the nonlinearity $f$ is given by \eqref{3D}. Then, for any radially symmetric data $(u_0,u_1)\in H^{1}\times L^{2}$, the
initial value problem \eqref{eq a1}-\eqref{eq aC.D} has a unique global solution $u\in {\mathcal C}(\R,H^{1})\cap
{\mathcal C}^{1}(\R,L^{2}).$ 
Moreover, $u$ satisfies the energy identity \eqref{eq a11}.
\end{thm}

\vspace{1.5cm}

\begin{rems}
 ~{\rm
 \begin{itemize}
 \item[(i)] The condition $1 + b \leq p < 4 + b\,(\leq 5 + 2b)$ guarantees the well-definedness of the energy, specifically that {$|x|^b |u|^{p+1} \in L^1(\mathbb{R}^3)$} whenever {$u \in H^1(\mathbb{R}^3)$ is radially symmetric}. This follows directly from Lemma \ref{GNI} below.
     \item[(ii)] We will focus exclusively on the case $3+b<p<4+b$. In the range $1+b<p\leq 3+b$, the local well-posedness can be readily obtained by employing energy methods.
     \item[(iii)] {The restriction $ p < 4 + b $ arises as a technical constraint within our methodology. Nevertheless, we conjecture that global well-posedness in the energy space holds for the full range $ 1 + b < p \leq 5 + 2b $. }
 \end{itemize}}   
\end{rems}

We now turn to the {\it linearization} as introduced in Definition \ref{add}. We consider, in ${\mathbb{R}}_{t}\times{\mathbb{R}}_{x}^{N}$ ($N=2, 3$),
 the linear  equation
\begin{equation}\label{eq a210} \Box v+\bm\,v=0,
\end{equation}
where $\Box:=\partial_{t}^{2}-\Delta_{x}$. Let $(\varphi _n,\psi_n)$ be a sequence in $H^1
\times L^2$ such that
\begin{equation}
(\varphi _n,\psi_n)\rightharpoonup (0,0)\quad\hbox{in}\quad
H^1\times L^2,\label{eq a19}
\end{equation}
\begin{equation}
\exists\; R_0>0,\quad\forall\; n,\quad \hbox{supp}(\varphi
_n)\cup\hbox{supp}(\psi_n)\subset\{
\;| x|\leq R_0\; \}.\label{eq a20}
\end{equation}
{Our main result regarding {\it linearization} can be stated as follows.}
\begin{thm}\label{at2}
Assume $b>0$ and  that the nonlinearity $f$ satisfies either \eqref{2D} when $N=2$  with $ b\leq 1/2, \,{\bm=1}$ or \eqref{3D} when $N= 3$ with  $1+b<p<4+b,\, {\bm=0}$.
We denote by $(u_n),\:(\mbox{respectively }(v_n))$ the
sequence of finite energy {solutions} to \eqref{eq a1},
$(\mbox{respectively } \eqref{eq a210})$ such that
\begin{equation*}
(u_n(0),\partial_t{u_n}(0))=(v_n(0), \partial_t{v_n}(0))=(\varphi _n,\psi_n).
\end{equation*}
Then, for any $T>0$, we have
\begin{equation*}
\sup_{t\in[0,T]}E_0(u_n-v_n,t)\longrightarrow 0\quad\mbox{as}
\quad n\rightarrow\infty,\end{equation*}
where 
\begin{equation}
    \label{E-0}
    E_0(w,t)= \frac{1}{2}\Big(\|\partial_{t}w(t)\|_{L^2}^2+\|\nabla\,w(t)\|_{L^2}^2+\bm \|w(t)\|_{L^2}^2\Big).
\end{equation}
\end{thm}

\vspace{0.7cm}

\begin{rems}
~
{\rm \begin{itemize}
    \item[(i)] {A sequence $(u_n)$ of solutions to \eqref{eq a1} is termed {\it finite energy} or {\it energy-bounded} if $\displaystyle\sup_{n\in\N}E(u_n)<\infty$, where the energy $E$ is defined by \eqref{energy}.}
    \item[(ii)] The aforementioned result justifies the sub-critical nature of the nonlinearities {considered in Theorem \ref{main2D} and Theorem \ref{main3D}}.
    \item[(iii)] The case $b=0$ in two spatial dimensions was investigated in \cite{OMTS}.
    \item[(iv)] The case $b=0$ in three spatial dimensions was examined by G\'erard in \cite{P.GR}.
    \item[(v)] For both global existence and linearization, one can consider more general defocusing nonlinearities $f:\mathbb{R}\to\mathbb{R}$ satisfying either
    \begin{equation*}
        f(0)=0,\quad |f(u)-f(v)|\lesssim |u-v|\left(|u|^{p-1}+|v|^{p-1}\right)\quad\mbox{if}\quad N=3,
    \end{equation*}
    or, for $N=2$,
    \begin{equation*}
        \left\{\begin{array}{ll} &|f(u)|\,\underset{u\to 0}{\sim}\,|u|^\kappa, \quad \kappa>1,\\
&\displaystyle\lim_{|u|\to\infty} |f(u)|\,{\rm e}^{-\alpha u^2}=0, \quad \forall \,\,\alpha>0.
\end{array}\right.
    \end{equation*}
\end{itemize}
    }
    {\rm Nevertheless, to ensure clarity in our presentation, we restrict our analysis to nonlinearities specified by \eqref{3D} and \eqref{2D}.}
\end{rems}
\begin{Rem}
{{\rm The results presented above can be extended to higher dimensions $ N \geq 3 $. However, for clarity and to avoid unnecessary technical complications, we restrict our focus to dimensions $ N = 2, 3 $ in this work.}}
\end{Rem}

The remainder of this paper is organized as follows: Section \ref{S2} presents the key technical tools used in our analysis. Section \ref{S3} focuses on establishing the global existence and uniqueness of solutions in the energy space. In Section \ref{S4}, we address linearization and provide the proof of Theorem \ref{at2}. Throughout the paper, we adhere to the standard convention of representing positive constants with $C$, which may vary between expressions. Additionally, we utilize the shorthand $A\lesssim B$ to denote the inequality $A\leq CB$ for some $C>0$.

\section{Useful tools \& Auxiliary results}
\label{S2}
In this section, we will introduce some tools and auxiliary results crucial for establishing our main results. To estimate the nonlinear term in 2D, we will utilize  the following Moser-Trudinger inequalities.
\begin{thm}\cite{Souza1}\label{theo-MT}
		Let $0\leq \beta <2$ and $0<\alpha<2\pi(2-\beta)$. Then, there exists a positive constant $C=C(\alpha,\beta)$ such that
		\begin{equation}
		\label{WMT1}
		\int_{\R^2}\frac{{\rm e}^{\alpha| u(x)|^2}-1}{| x|^\beta}dx\leqslant C\int_{\R^2}\frac{| u(x)|^2}{| x|^\beta}dx,
		\end{equation}
		for all $u\in H^1(\R^2)$ with $\| \nabla u\|_{L^2(\R^2)}\leqslant 1$.
	\end{thm}
	We point out that $\alpha=2\pi(2-\beta)$ becomes admissible in
	(\ref{WMT1}) if we require $\|u\|_{H^1(\R^2)}\leq1$ instead of
	$\|\nabla u\|_{L^2(\R^2)}\leq1$, where
 $$
 \|u\|_{H^1}^2=\|u\|_{L^2}^2+\|\nabla u\|_{L^2}^2.
 $$
 More precisely, we have
	\begin{thm}\cite{Souza2}\label{MTT}
		Let $0\leq \beta <2$. We have
		\begin{equation}
		\label{WMT2}
		\mathcal{A}_\beta:=\sup_{\| u\|_{H^1(\R^2)}\leqslant 1}\int_{\R^2}\frac{{\rm e}^{2\pi(2-\beta)| u(x)|^2}-1}{| x|^\beta}dx<\infty.
		\end{equation}
  Moreover, the inequality \eqref{WMT2} is sharp in the sense that for any $\epsilon >0$, we have
  \begin{equation*}
		\sup_{\| u\|_{H^1(\R^2)}\leqslant 1}\int_{\R^2}\frac{{\rm e}^{(2\pi(2-\beta)+\epsilon)| u(x)|^2}-1}{| x|^\beta}dx=\infty.
		\end{equation*}
	\end{thm}
	\begin{rems}
	    ~{\rm \begin{itemize}
	        \item[(i)] The case $\beta=0$ has been previously examined in \cite{AD.TA}. Therefore, inequality \eqref{WMT1} can be regarded as an extension of \cite[Theorem 0.1]{AD.TA}, particularly in the context of incorporating a singular weight, that is $\beta>0$.
         \item[(ii)] Inequality \eqref{WMT2} was previously established in \cite{RU} for the case $\beta=0$.
	    \end{itemize}}
\end{rems}
For a function $F: \R\to\R$ with subcritical growth at infinity in 2D and appropriate behavior near zero, it can be shown that 
$|x|^b F(u)\in L^1(\R^2)$ whenever 
$u\in H^1(\R^2)$  is radially symmetric. More specifically:
\begin{Lem}
    \label{L-1}
    Let $F:\R\to\R$ be a continuous function that satisfies the following conditions:
    \begin{equation}
        \label{A-1}
        |F(u)|\underset{u\to 0}{\sim} |u|^\kappa, \quad \kappa>2,
        \end{equation}
        \begin{equation}
        \label{A-2}
        \lim_{|u|\to\infty} |F(u)|\,{\rm e}^{-\alpha u^2}=0, \quad \forall \,\,\alpha>0.
   \end{equation}
Then, for any radially symmetric $u\in H^1(\R^2)$, we have $|x|^b F(u)\in L^1(\R^2)$ provided that $0\leq b\leq \kappa/2-1$.
\end{Lem}
\begin{proof}[Proof of Lemma \ref{L-1}]
Consider a radially symmetric function $u\in H^1(\R^2)$. By expressing  $|x|^b |F(u)|= \left(|x|^{1/2} |u|\right)^{2b}\, |u|^{-2b}|F(u)|$ and applying the radial Sobolev inequality \eqref{est-Strauss}, it is sufficient to prove that $|u|^{-2b}|F(u)|\in L^1$. Using conditions \eqref{A-1}-\eqref{A-2}, we observe that for any $\alpha>0$, the following holds:
\begin{equation*}
|u|^{-2b}|F(u)|\lesssim_\alpha\, |u|^{\kappa-2b}+\left({\rm e}^{\alpha u^2}-1\right).
\end{equation*}
Given that $\kappa-2b\geq 2$ and ${\rm e}^{\alpha u^2}-1 \in L^1$ (see \cite[Lemma 3.5]{BBMS}), the proof is thus complete.
\end{proof}
By following the argument outlined in the proof of Lemma \ref{L-1}, we can easily derive the result stated below.
\begin{Lem}
    \label{Tech-2D}
    Let $0< b\leq  1$ and $\alpha>0$. There exists a positive constant $C$ such that, for any radially symmetric function $u\in H^1(\R^2)$, we have
    \begin{equation}
        \label{Ener-2D}
        \left||x|^b({\rm e}^u-1-u)\right|\leq C\|u\|_{H^1}^{2b}\left(|u|^{2(1-b)}+\left({\rm e}^{\alpha u^2}-1\right)\right).
    \end{equation}
\end{Lem}
\begin{proof}
 {Observe that 
$$
\left||x|^b({\rm e}^u - 1 - u)\right| = \left(|x|^{1/2}|u|\right)^{2b} |u|^{-2b}\left|{\rm e}^u - 1 - u\right|.
$$
Since $ |u|^{-2b}\left|{\rm e}^u - 1 - u\right| \,\underset{u \to 0}{\sim}\, \frac{1}{2}|u|^{2(1 - b)} $ and, for $ |u| \geq 1 $, 
$$
|u|^{-2b}\left|{\rm e}^u - 1 - u\right| \leq \left|{\rm e}^u - 1 - u\right| \lesssim_{\alpha} {\rm e}^{\alpha u^2} - 1 \quad \text{for any } \alpha > 0,
$$
the desired estimate \eqref{Ener-2D} follows by applying \eqref{est-Strauss} below.}   
\end{proof}
The above inequalities will be often combined with the following 2D energy estimate.
\begin{Prop}$($Energy Estimate$)$
\label{aprop2} {There exists a constant $C_0$ such that, for any $T>0$ and any $u\in C([0,T]; H^1)\cap C^1([0,T]; L^2)$, the following is true:}
\begin{equation}\label{eq aen}
\|u\|_T \leq C_0\,\Big
[\|u(0)\|_{H^{1}}+\|\partial_{t}
u(0)\|_{L^{2}} +\|\Box
u+u\|_{L^{1}_{T}(L^{2})}\Big],
\end{equation}
where 
\begin{equation}
    \label{norm-T}
    \|u\|_T:=\displaystyle\sup_{t\in [0,T]}\Big(\|u(t)\|_{H^{1}(\mathbb{R}^2)}
+\|\partial_{t}u(t)\|_{L^{2}( {\mathbb{R}}^{2})}\Big).
\end{equation}
\end{Prop}

We will also review the following Strichartz estimates. To do this, we introduce the concept of admissible pairs in three dimensions ($N=3$).
\begin{Def}
\label{STP}
Let $N=3$. A pair $(q,r)$ is said to be $H^1$ wave-admissible if
\begin{equation}
\label{STP1}
\frac{1}{q}+\frac{3}{r}=\frac{1}{2},\;\;2< q\leq\infty,\;\; 6\leq r<\infty.
\end{equation}
\end{Def}
For example, $(\infty,6)$, $(5,10)$, $(4,12)$, and $(8,8)$ are admissible pairs.
\begin{Prop}[Strichartz Estimates]
\label{ST}
{Consider $N=3$, and let $(q,r)$ be $H^1$ wave-admissible as defined in \eqref{STP1}. For $T>0$, there exists a constant $C=C(q)>0$ such that, for any $v\in C([0,T]; H^1)\cap C^1([0,T]; L^2)$, the following holds:}
\begin{equation}
\label{StEs}
\|v\|_{L^q([0,T]; L^r(\R^N))}\leq C\left(\|\nabla v(0)\|_{L^2}+\|\partial_t v(0)\|_{L^2}+\|\Box v\|_{L^{q'}([0,T]; L^{r'}(\R^N))}\right).
\end{equation}
\end{Prop}

Estimating the nonlinear term in $L^1_T(L^2)$ presents a significant challenge due to the weight $|x|^b$. This difficulty can be addressed by using the following radial Sobolev inequality.
\begin{Lem}[\cite{Strauss}]
    \label{Rad-Sob} For $N\geq 2$ and any radially symmetric function $\phi \in H^1(\mathbb{R}^N)$, the following holds:
 \begin{align} \label{est-Strauss}
		|x|^{\frac{N-1}{2}} |\phi(x)| \leq C \|\nabla \phi\|^{\frac{1}{2}}_{L^2} \|\phi\|^{\frac{1}{2}}_{L^2}.
		\end{align}
\end{Lem}
We also recall the following sharp Gagliardo-Nirenberg inequality.
\begin{Lem}
    \label{GNI}
    Given $N\geq 2$, $\theta>0$, $\lambda\geq 2+\frac{2\theta}{N-1}$, and $\lambda\leq \frac{2N+2\theta}{N-2}$ for $N\geq 3$, the sharp Gagliardo-Nirenberg inequality stated below holds:
    \begin{equation}
        \label{SGNI}
        \int_{\R^N}\, |x|^\theta\,|\varphi(x)|^\lambda\,dx\leq K_{opt}\,\|\varphi\|_{L^2}^A\, \|\nabla\varphi\|_{L^2}^B,
    \end{equation}
    where $\varphi\in H^1(\mathbb{R}^N)$ is radially symmetric and
    \begin{equation*}
        A=2+\theta-\frac{N-2}{2}\lambda,\quad B=\frac{N}{2}(\lambda-2)-\theta.
    \end{equation*}
\end{Lem}
\begin{Rem}
    {\rm The proof of Lemma \ref{GNI} can be found in several references, such as \cite{DMS, Fuz}. In particular, it has been established that the upper bound $\frac{2N+2\theta}{N-2}$ is optimal for $N\geq 3$, as it corresponds to the energy critical regularity. Additionally, it has been shown in \cite[Lemma 2.5]{DMS} that the lower bound $2+\frac{2\theta}{N-1}$ is also optimal for the Gagliardo-Nirenberg inequality \eqref{SGNI}. }
\end{Rem}
By utilizing Lemma \ref{GNI}, we derive a crucial estimation that will aid in our analysis of well-posedness in three spatial dimensions.
\begin{Lem}
    \label{Interp-Est}
    {Given $N=3$, $b>0$, $T>0$, $3+b<p<4+b$, $u\in L^\infty_T(H^1)$, and $v\in L^q_T(L^r)$, the following inequality holds}
    \begin{equation}
        \label{Int-Est}
        \left\||x|^b |u|^{p-1}v\right\|_{L^1_T(L^2)}\leq C T^{\frac{5+b-p}{2}}\, \|u\|_{L^\infty_T(H^1)}^{p-1}\,\|v\|_{L^q_T(L^r)},
    \end{equation}
    where $(q,r)$ is given by \eqref{qr} below.
\end{Lem}
\begin{proof}[Proof of Lemma \ref{Interp-Est}] By employing H\"older's inequality with $\gamma=\frac{6}{p-1-b}$ and $r=\frac{6}{4+b-p}$, followed by the Gagliardo-Nirenberg inequality \eqref{SGNI} with $\theta=b\gamma$ and $\lambda=(p-1)\gamma$, we derive
\begin{equation*}
    \begin{split}
        \left\||x|^b |u|^{p-1}v\right\|_{L^2}&\leq \left\||x|^b |u|^{p-1}\right\|_{L^\gamma}\left\|v\right\|_{L^r}\\
        &\lesssim \|u\|_{H^1}^{p-1}\, \|v\|_{L^r}.   \end{split}
\end{equation*}
Next, by applying H\"older's inequality in the temporal domain $[0,T]$ with $q=\frac{2}{p-b-3}$, we deduce that
\begin{equation*}
    \begin{split}
        \left\||x|^b |u|^{p-1}v\right\|_{L^1_T(L^2)}
        &\lesssim \left\|\|u(t)\|_{H^1}^{p-1}\right\|_{L^{q'}_T}\, \|v\|_{L^q_T(L^r)}\\
        &\lesssim T^{\frac{5+b-p}{2}}\,\|u\|_{L^\infty_T(H^1)}^{p-1}\,\|v\|_{L^q_T(L^r)}.
        \end{split}
\end{equation*}
This concludes the proof of Lemma \ref{Interp-Est}.
\end{proof}
\begin{Rem}
    {\rm As we will see shortly, the pair $(q, r)$ is admissible as it satisfies the condition \eqref{STP1}.}
\end{Rem}
To address the linearization, we need to consider some important facts regarding convergence in measure.
\begin{Def}\label{ad1d}
{Let $\Omega$ be an open subset of $\mathbb{R}^d$, and let $u_n: \Omega\to\R$ be a sequence of measurable functions.} We say that a sequence $(u_n)$ converges in measure to $u$ on $\Omega$ if
$$\forall\varepsilon >0,\quad \Big|\left\{x\in\Omega,|u_n(x)-u(x)|>\varepsilon\right\}\Big|\longrightarrow 0\quad\mbox{as}\quad
n\longrightarrow \infty,$$
where $|\cdot|$ is the Lebesgue measure on $\mathbb{R}^d$.
\end{Def}
\begin{Lem}\label{armf}
{Let $\Omega$ be an open subset of $\mathbb{R}^d$, $1\leq q<\infty$, and let $u_n: \Omega\to\R$ be a sequence in $L^q(\Omega)$.}
\begin{enumerate}
\item[1)] If $u_n\longrightarrow u\mbox{ in }L^q(\Omega)\,(1\leq q <\infty),\;\; \mbox{ then }\;\; u_{n}\longrightarrow u\;\; \mbox{ in measure on }\Omega$  .
\item[2)] For any continuous function $G: \R\to\R$
$$ u_n\to u\mbox{ in measure on }\Omega\quad \Longrightarrow\quad G(u_n)\longrightarrow G(u)\;\mbox{ in measure on }\;\Omega.$$
\end{enumerate}
\end{Lem}
The proof of Lemma \ref{armf} is straightforward; hence, we omit the details here.
\section{Global well posedness}
\label{S3}

This section is dedicated to proving Theorems \ref{main2D} and \ref{main3D}, which concern global well-posedness in the energy space. The strategy of proofs can be summarized as follows.
To demonstrate local existence of solutions, we employ a classical fixed-point argument. We then establish uniqueness in an unconditional sense. Finally, we prove that the maximal solution is global.

\subsection{Proof of Theorem \ref{main2D}}
\subsubsection{Local Existence}
Let us start by proving the existence. For $T>0$, we denote by  
$$\bX_T:=\mC([0,T],H^1(\R^2)) \cap  \mC^1([0,T],L^2(\R^2)) $$
endowed by the norm  $\|.\|_T$ as defined in \eqref{norm-T}.
For $R>1+\|u_0\|_{H^1}+\|u_1\|_{L^2}$, let $B_T(R)$ denotes the ball in $\bX_T$ of radius $R$ and centered at the origin. We define the map $\Phi$ on the ball $B_T(R)$ by $v \longmapsto\Phi (v)=\tilde{v} $
where $\tilde{v}$ solves 
\begin{equation*}
\Box\tilde{v}+\tilde{v}=-|x|^b f(v+v^0) , \quad\textbf{}
(\tilde{v}(0),\tilde{v}_t(0))=(0,0),
\end{equation*}
and $v^0$ stands for the solution of the  homogeneous equation 
\begin{equation*}
\Box{v^0}+ v^0=0, \quad
(v^0(0),\partial_tv^0(0)=(u_0, u_1).
\end{equation*}
Here we suppose that the nonlinearity $f$ is given by \eqref{2D}.  {Observe that, by combining Lemma \ref{Tech-2D} with the Moser-Trudinger inequality \eqref{WMT1}, it is clear that $|x|^b f(v+v^0)$ belongs to $L^1_T(L^2)$ for any  $v\in \bX_T$. Consequently, the existence of the solution $\tilde{v}$ follows directly.}
We will prove that the map $\Phi$ is a contraction from $B_T(R)$ into itself for $T, R>0$ suitably chosen.
We start by showing {that $\Phi$ maps $B_T(R)$ to itself}. Let $v\in B_T(R)$. Applying the energy estimate \eqref{eq aen} one gets 
\begin{equation}
    \label{Phi-v-T}
    \|\Phi (v) \|_{T}=\| \tilde{v}\|_T \leq C_0 \left\| |x|^b f(v+v^0)\right\|_{L^1_T(L^2)}. 
\end{equation}

For $\alpha >0$, we obtain by applying Lemma \ref{Tech-2D} that 
\begin{equation*}
\int_{\R^2}|x|^{2b}|f(v+v^0)|^2dx\leq C\|v+v^0\|_{H^1}^{4b} \int_{\R^2}\left(|v+v^0|^{4(1-b)}+(e^{2\alpha (v+v^0)^2}-1)\right)dx.
\end{equation*}
Note that by the energy estimate \eqref{eq aen} and the fact that $v\in B_T(R)$, we get
\begin{equation*}
    \|v+v^0\|_{H^1}\leq \Lambda:=R+C_0(\|u_0\|_{H^1}+\|u_1\|_{L^2}).
\end{equation*}
Therefore, considering \eqref{WMT2} (with $\alpha\leq \frac{2\pi}{\Lambda^2}$) and noting that $4(1-b)\geq 2$ and $\Lambda>1$, it follows that 
\begin{equation}
\label{v-f}
\begin{split}
    \| |x|^bf(v+v^0)  \|_{L^2}^2&\leq C \Lambda^{4b}\left(\|v+v^0\|_{L^{4(1-b)}}^{4(1-b)}+\int_{\R^2}({\rm e}^{2\alpha\Lambda^2(\frac{v+v_0}{\Lambda})^2}-1)dx\right)\\
    &\leq C\left(\Lambda^4+\Lambda^{4b}\right)\\
    &\leq C \Lambda^4.
    \end{split}
\end{equation}
From \eqref{v-f} and  \eqref{Phi-v-T}, it follows straightforwardly that 
\begin{equation*}
     \|\Phi (v) \|_{T}\leq C_0 C\Lambda^2 T\leq R,
\end{equation*}
provided  $T>0$ is chosen to be sufficiently small. This implies that $\Phi(B_T(R)\subset B_T(R)$.\\
It remains to  show that $\Phi$ is a contraction mapping on 
$B_T(R)$. To this end, let $v_1, v_2\in B_T(R)$. Define $\tilde{v}:=\Phi(v_1)-\Phi(v_2)=\tilde{v}_1-\tilde{v}_2$, which satisfies
\begin{equation*}
\Box \tilde{v}+\tilde{v}=-|x|^b\left( f(v_1+v^0)-f(v_2+v^0)\right), \quad
(\tilde{v}(0),\tilde{v}_t(0))=(0, 0).
\end{equation*}
By expressing $f(v_1+v^0)-f(v_2+v^0)$ as $(v_1-v_2)({\rm e}^{\bar{v}}-1)$, where $\bar{v}=(1-\theta)(v_1+v^0)+\theta(v_2+v^0)$, and applying the energy estimate \eqref{eq aen}, we obtain
\begin{equation}
\label{Cont-2D-1}
\|\Phi(v_1)-\Phi(v_2)\|_T\leq C_0 \left\||x|^b(v_1-v_2)({\rm e}^{\bar{v}}-1)\right\|_{L^1_T(L^2)}.
\end{equation}
First, we estimate the $L^2$ norm with respect to the spatial variable. Applying the elementary inequality $({\rm e}^s-1)^2 \leq {\rm e}^{2s}-1-2s$ for all $s\in\R$ and utilizing a convexity argument, we obtain
\begin{equation}
    \label{Cont-2D-2}
    \int_{\R^2} |x|^{2b}(v_1-v_2)^2\left({\rm e}^{\bar{v}}-1\right)^2 dx\leq  \mathbf{I_1}+\mathbf{I_2},
\end{equation}
where
\begin{equation*}
    \mathbf{I_j}=\int_{\R^2} |x|^{2b}(v_1-v_2)^2\left({\rm e}^{2(v_j+v^0)}-1-2(v_j+v^0)\right) dx,\quad j=1,2.
\end{equation*}
We estimate $\mathbf{I_j}$ as follows
\begin{equation*}
    \begin{split}
        \mathbf{I_j}&=\int_{\R^2} \left(|x|^{1/2}|v_1-v_2|\right)^2 |x|^{2b-1}\left({\rm e}^{2(v_j+v^0)}-1-2(v_j+v^0)\right) dx\\
        &\leq C\|v_1(t)-v_2(t)\|_{H^1}^2\int_{\R^2}\frac{{\rm e}^{4\alpha(v_j+v^0)^2}-1}{|x|^{1-2b}}dx\\
        &=C\|v_1(t)-v_2(t)\|_{H^1}^2\int_{\R^2}\frac{{\rm e}^{4\alpha\Lambda^2(\frac{v_j+v^0}{\Lambda})^2}-1}{|x|^{1-2b}}dx,
    \end{split}
\end{equation*}
where we have applied the radial Sobolev inequality \eqref{est-Strauss} and the observation that
\begin{equation*}
\mathbf{K}_\alpha:=\sup_{s\in\R}\left(\frac{{\rm e}^s-1-s}{{\rm e}^{\alpha s^2}-1}\right)<\infty, \quad \forall \, \alpha>0.
\end{equation*}

Choosing $\alpha\leq \frac{\pi(1+2b)}{2\Lambda^2}$ and applying the singular Moser-Trudinger inequality \eqref{WMT2}, we infer that
\begin{equation}
    \label{Cont-2D-4}
     \mathbf{I_j}\leq C \|v_1(t)-v_2(t)\|_{H^1}^2.
\end{equation}
Plugging estimates \eqref{Cont-2D-1}, \eqref{Cont-2D-2}, and \eqref{Cont-2D-4} together, we find that 
\begin{equation*}
    \|\Phi(v_1)-\Phi(v_2)\|_T\leq C T\|v_1-v_2\|_T.
\end{equation*}
By choosing $T>0$ sufficiently small, we establish that $\Phi$ is a contraction mapping. Consequently, the Banach fixed point theorem guarantees the existence of a unique fixed point $v\in B_T(R)$. Therefore, $u:=v+v^0$  constitutes the desired local solution within the energy space. Notably, this result also ensures uniqueness within the class $v_0+B_T(R)$.
\begin{Rem}
    {\rm By noting that $\mathbf{K}_\alpha\lesssim \frac{1}{\alpha}$ for $\alpha\geq 1$, we can conclude that the time of local existence depends solely on the size of the initial data. This insight enables us to establish global existence, as described below.}
\end{Rem}

\subsubsection{Global existence }
Let $u$ represent the unique maximal solution defined over the time interval $[0, T^*)$, where $0 < T^* \leq \infty$ denotes the maximum lifespan.  We argue by contradiction, assuming that $T^*$ is finite. Consider, for $0<t_0<T^*$, the following initial value problem: 
\begin{equation}
\label{GWP-2D}
\begin{split}
\left\{
\begin{array}{cllll}
\Box v+v&=&-|x|^b f(v),\\
(v(t_0), v_t(t_0))&=&(u(t_0),u_t(t_0)).\\
\end{array}
\right.
\end{split}
\end{equation}
According to the local theory, there exists a non-negative $\tau$ and a unique solution 
$v$ to \eqref{GWP-2D} on the time interval $[t_0,t_0+\tau]$. 
Since the time of local existence depends only on the size of the initial data and by the conservation of energy, $\tau$ is independent of $t_0$. By choosing $t_0$ sufficiently close to $T^*$ such that $T^*-t_0<\tau $, we can extend $u$ beyond the maximal time  $T^*$. This contradicts the definition of $T^*$, leading us to conclude that $T^*=\infty$ as desired.

\subsection{Proof of Theorem \ref{main3D}}
Here, we assume that $N=3$, $b>0$, and $3+b<p<4+b$. {Let $\mathbf{D}:=\left(\partial_t, \nabla\right)$, $\mathcal{E}_0:= \|\mathbf{D}u(0)\|_{L^2}$ denote the initial kinetic energy, and the pair $(q,r)$ be given by}
\begin{equation}
\label{qr}
\left(q,r\right)=\left(\frac{2}{p-b-3}, \frac{6}{4+b-p}\right).
\end{equation}

One can easily verify that $(q,r)$ is an admissible pair in the sense that it satisfies \eqref{STP1}. We will employ a standard fixed-point argument. For this purpose, we introduce the following Banach space 
$$
{\mathbf E}_T={\mathcal C}([0,T];\dot{H}^1)\cap{\mathcal
C}^1([0,T]; L^2)\cap L^q([0,T]; L^r) $$
endowed with the norm $$ \|u\|_{T}:=\|\mathbf{D}u\|_{L^\infty_T(L^2)}+\|u\|_{L^q_T(L^r)},$$
where $T>0$ and $(q,r)$ is given by \eqref{qr}.
For a positive real number $R$, we
denote by $B_T(R)$   the ball in ${\mathbf E}_T$
of radius $R$ and centered at the origin. On the ball
$B_T(R) $, we define the map $\Phi$ by
\begin{equation*}
v\longmapsto\Phi(v):=\tilde{v},
\end{equation*}
where
\begin{equation}
\label{Phivvv}
\Box\tilde{v}=-|x|^b|v|^{p-1}v, \quad
(\tilde{v}(0),\tilde{v}_t(0))=(u_0, u_1).
\end{equation}
Our aim now is to show that, if  $R$ and $T$ are suitably chosen, the map $\Phi$ is well defined from $B_T(R)$ into itself and it is a contraction. From \eqref{Int-Est} we see that $|x|^b|v|^{p-1}v$ belongs to $L^1_T(L^2)$ as long as $v\in {\mathbf E}_T$. It follows that for every $v\in {\mathbf E}_T$, the Cauchy problem \eqref{Phivvv} has a unique solution $\tilde{v}$ in ${\mathbf  E}_T$. Moreover, if $v\in B_T(R)$, then by \eqref{StEs} and \eqref{Int-Est}, we have
\begin{equation}
\label{Stab}
\begin{split}
\|\Phi(v)\|_T&\leq C\left(\mathcal{E}_0+\||x|^b|v|^{p-1}v\|_{L^1_T(L^2)}\right)\\
&\leq  C\left(\mathcal{E}_0+T^{\frac{5+b-p}{2}}\,\|v\|_T^{p}\right)\\
&\leq C\left(\mathcal{E}_0+T^{\frac{5+b-p}{2}}\,R^{p}\right).
\end{split}
\end{equation}
Choosing $R=2C \mathcal{E}_0$ and owing to \eqref{Stab}, we see that $\Phi(B_T(R))\subset B_T(R)$ as long as $T^{\frac{5+b-p}{2}}\,R^{p}\leq \mathcal{E}_0$. The latter condition reads in terms of $T$ as 
\begin{equation}
\label{T1}
T\lesssim \mathcal{E}_0^{-\kappa},
\end{equation}
where $\kappa:=\frac{2(p-1)}{5+b-p}>0$ since $1<p<5+b$.\\

It remains to  show that $\Phi$ is a contraction mapping on 
$B_T(R)$. Observe that given $v_1$ and $v_2$ in $B_T(R)$, $w:=\Phi(v_1)-\Phi(v_2)=\tilde{v}_1-\tilde{v}_2$ satisfies
\begin{equation*}
\Box w=-|x|^b (|v_1|^{p-1}v_1-|v_2|^{p-1}v_2), \quad
(w(0),w_t(0))=(0, 0).
\end{equation*}
Thanks to  Strichartz estimate \eqref{StEs}, we get
$$
\|\Phi(v_1)-\Phi(v_2)\|_T\leq C \left\||x|^b (|v_1|^{p-1}v_1-|v_2|^{p-1}v_2)\right\|_{L^1_T(L^2)}.
$$
Arguing as above and using \eqref{Int-Est} coupled with the elementary inequality $$\left||v_1|^{p-1}v_1-|v_2|^{p-1}v_2\right|\leq\, p\,|v_1-v_2|\left(|v_1|^{p-1}+|v_2|^{p-1}\right),$$
we end up with
\begin{equation}
\label{Cont}
\begin{split}
\|\Phi(v_1)-\Phi(v_2)\|_T&\leq\,C\, \left\||x|^b|v_1-v_2|\left(|v_1|^{p-1}+|v_2|^{p-1}\right)\right\|_{L^1_T(L^2)}\\
&\leq\, C\,T^{\frac{5+b-p}{2}}\, \left(\|v_1\|_{L^\infty_T(H^1)}^{p-1}+\|v_2\|_{L^\infty_T(H^1)}^{p-1}\right)\,\|v_1-v_2\|_{L^q_T(L^r)}\\
&\leq\, C\,T^{\frac{5+b-p}{2}}\,\left(\mathcal{E}_0+R\right)^{p-1} \left\|v_1-v_2\right\|_{T}.
\end{split}
\end{equation}
If $T$ satisfies \eqref{T1}, then \eqref{Cont} tell us that $\Phi$ is a contraction mapping on $B_T(R)$. This finishes the local existence part. The unconditional uniqueness can be done easily by using similar argument. Finally, the maximal solution is global in the defocusing regime due to the energy conservation and  the blow-up alternative.


\section{Linearization}
\label{S4}
This section is devoted to  the proof of Theorem \ref{at2}. Hereafter, $T$ stands for a fixed positive time, and $(\varphi_n,\psi_n)$ a sequence in $H^1\times L^2$ satisfying \eqref{eq a19} and \eqref{eq a20}.
We denote by $(u_n)$, (respectively $(v_n)$) the sequence of finite energy solutions to \eqref{eq a1}, (respectively \eqref{eq a210}) such that 
$$(u_n(0),\partial_tu_n(0))= (v_n(0),\partial_tv_n(0))=(\varphi_n,\psi_n).$$
\subsection{The 2D case}
Recall that in this context, we have $f(u) = -\left({\rm e}^u-1-u\right)$, leading to $F(u) = -\left({\rm e}^u-1-u-\frac{u^2}{2}\right)$. An essential component of the proof hinges on the following lemma.
\begin{Lem}
\label{Key2D}
Up to extraction, we have 
\begin{equation*}
    u_n \rightharpoonup 0 \quad \mbox{in}\quad H^1 (]0,T[\times \R^2),
\end{equation*}
   and 
   \begin{equation*}
    u_n \longrightarrow 0 \quad \mbox{in}\quad L^2(]0,T[\times \R^2).
\end{equation*}
\end{Lem}

\begin{proof}[Proof of Lemma \ref{Key2D}]
    Let 
    $$
    M:=\sup_n \left(\|\varphi_n\|_{H^1}^2+\|\psi_n\|_{L^2}^2\right)\quad\mbox{and}\quad 0<\alpha <\frac{4\pi}{M}.$$
    By the energy conservation  and Lemma \ref{L-1} with $\kappa=3$, we have
   \begin{eqnarray*}
 E(u_n,t)&=&\frac{1}{2}(\|\varphi_n\|_{H^1}^2+\|\psi_n\|_{L^2}^2)+\int_{\R^2}|x|^b F(\varphi_n)dx\\
 &\leq& \frac{M}{2}+CM^b\int_{\R^2}\left(|\varphi_n|^{3-2b}+\left({\rm e}^{\alpha \varphi_n^2}-1\right)\right)dx\\
 &\leq&\frac{M}{2}+C M^{3-b}+C M^b.
  \end{eqnarray*}
  Therefore, once more through energy conservation, we obtain for $0 \leq t \leq T$,
  \begin{equation}
      \label{u-n-bound}
      \|\nabla u_n(t)\|_{L^2}^2+\|u_n(t)\|_{L^2}^2\leq 2E(u_n,t) \leq C_M:=M+2C M^{3-b}+2C M^b.
  \end{equation}
It follows that
\begin{eqnarray*}
\|u_{n}\|^{2}_{H^{1}(]0,T[\times{\mathbb{R}}^{2})}&=&\int_{0}^{T}\Big(\|u_{n}(t)\|_{H^1(\mathbb{R}^2)}^{2}+\|\partial_{t}u_{n}(t)\|^{2}_{L^2(\mathbb{R}^2)}\Big)dt\\
&\leq&2TE(u_n,t)\\
&\leq& 2T \left(\frac{M}{2}+C M^{3-b}+C M^b\right).
\end{eqnarray*}
Hence, after extraction, the sequence $(u_n)$ possesses a weak limit $u$ in $H^1(]0,T[\times \R^2)$. To finalize the proof, we will show that $u$ satisfies the following initial value problem
\begin{equation*}
\Box u + |x|^b f(u) = 0, \quad
(u(0, \cdot), \partial_t u(0, \cdot)) = (0, 0),
\end{equation*}
which ultimately leads to $u = 0$ through the application of a standard uniqueness argument. 

Although we deal here with a spatial growing nonlinearity, the remainder of the proof uses Lemma \ref{L-1} and mimics the same steps performed in \cite[Lemma 4.3]{OMTS}, which specifically addresses the case $b = 0$. For brevity, we omit the detailed steps. Consequently, the proof of Lemma \ref{Key2D} is now complete.
    \end{proof}

We return to the proof of the 2D part in Theorem \ref{at2}. Setting $w_n=u_n-v_n$, we have 
\begin{equation*}
\Box w_n+w_n=-|x|^bf(u_n), \quad
\left(w_n(0,.),\partial_tw_n(0,.)\right)=(0,0).
\end{equation*}
By applying the energy estimate to $w_n$ over the interval $[0, T]$, and then utilizing Lemma \ref{Tech-2D} with $0 < \alpha < \frac{2\pi}{C_M}$ (where $C_M$ is defined in \eqref{u-n-bound}), along with \eqref{u-n-bound} and \eqref{WMT1} (with $\beta=0$), we derive for $t \in [0, T]$
\begin{equation}
    \label{Linear-2D}
    \begin{split}
        E_0(w_n,t)&\lesssim \||x|^bf(u_n)\|_{L^1_T(L^2)}^2\\
        &\lesssim T\, \||x|^bf(u_n)\|_{L^2([0,T]\times\R^2)}^2\\
        &\lesssim \|u_n\|_{L^{4(1-b)}}^{4(1-b)}+\int_0^T\,\int_{\R^2}\, \left({\rm e}^{2\alpha\,u_n^2}-1\right)dx\,dt\\
        &\lesssim \|u_n\|_{L^{4(1-b)}([0,T]\times\R^2)}^{4(1-b)}+\|u_n\|_{L^2([0,T]\times\R^2)}^2.
    \end{split}
\end{equation}
To conclude the proof, we observe that the sequence $(u_n)$ converges to $0$ in $L^2([0,T]\times\R^2)$ and is bounded in $H^1((0,T)\times\R^2)$. By interpolation, this convergence extends to $0$ in any Lebesgue space $L^q([0,T]\times\R^2)$ for $2 \leq q < 6$. Since $2 \leq 4(1-b) < 6$ due to $b \leq \frac{1}{2}$, we infer that the right-hand side in \eqref{Linear-2D} tends to zero as $n \to \infty$. This concludes the proof of the 2D component of Theorem \ref{at2}.
\subsection{The 3D case}
We  move now to the 3D part of Theorem \ref{at2}. In the following, we have $N=3$, $b>0$, $1+b<p<4+b$, and $T>0$.

As will become clear later on, the proof borrows some arguments from \cite[Theorem 2.3]{P.GR}. By defining $w_n=u_n-v_n$, we obtain 
\begin{equation*}
\Box w_n=-|x|^b|u_n|^{p-1}u_n, \quad
\left(w_n(0,.),\partial_tw_n(0,.)\right)=(0,0).
\end{equation*}
Applying energy estimate yields
\begin{equation*}
    \sup_{0\leq t\leq T}E_0(w_n,t)\lesssim \||x|^b|u_n|^{p-1}u_n\|_{L^1_T(L^2)}^2,
\end{equation*}
where $E_0(w,t)$ is defined by \eqref{E-0} with $\bm=0$. By \eqref{eq a16} and
the finite propagation speed, the solution $u_n$ is supported in $\displaystyle\Big\{(t,x);\; 0\leq t\leq T,\;  |x|\leq R_0+t\Big\}$. 

By the conservation of energy, \eqref{SGNI} and \eqref{eq a19}, we have
\begin{equation*}
    \begin{split}
        E(u_n,t)&=\frac{1}{2}\left(\|\nabla \varphi_n\|_{L^2}^2+\|\psi_n\|_{L^2}^2\right)+\frac{1}{p+1}\int_{\R^3}\,|x|^b |\varphi_n(x)|^{p+1}\,dx\\
        &\leq \frac{M}{2}+\frac{K_{opt}}{p+1} M^{\frac{p+1}{2}},
    \end{split}
\end{equation*}
where $M:=\displaystyle\sup_{n\in\N}\left(\|\nabla \varphi_n\|_{L^2}^2+\|\varphi_n\|_{L^2}^2+\|\psi_n\|_{L^2}^2\right)<\infty$. It follows that, for all $n\in\mathbb{N}$ and $t\in [0,T]$,
\begin{equation}
\label{u-n-bound-H1}
\|u_n(t)\|_{H^1}\leq M_T,
\end{equation}
where $M_T$ is a positive constant. Now, by utilizing the Strichartz estimate \eqref{StEs} for the admissible pair $(q,r)$ as defined in \eqref{qr}, along with Lemma \ref{Interp-Est}, we obtain
\begin{equation*}
 \|u_n\|_{L^q_T(L^r)}\lesssim \sqrt{M}+T^{\frac{5+b-p}{2}}\,\|u_n\|_{L^\infty_T(H^1)}^{p-1}\|u_n\|_{L^q_T(L^r)}.
\end{equation*}
Therefore, utilizing \eqref{u-n-bound-H1}, we deduce that $(u_n)$ is bounded in $L^q_T(L^r)$ for small $T>0$. By repeating this procedure iteratively, we establish the boundedness of $(u_n)$ in $L^q_T(L^r)$ for any $T>0$.
Now, the remaining part of the proof can be conducted analogously to the proof of \cite[Theorem 2.3]{P.GR}, incorporating Lemma \ref{armf}. Consequently, for the sake of brevity, the details are omitted. Thus, we thereby conclude the proof of Theorem \ref{at2} concerning linearization.

\end{document}